\documentclass{amsart}

\usepackage[T1]{fontenc}
\usepackage[utf8]{inputenc}
\usepackage{amsmath,amssymb,mathtools}
\usepackage[shortlabels]{enumitem}
\usepackage{csquotes}
\usepackage[hidelinks]{hyperref}
\hypersetup{
  pdftitle={Counterexamples to the Generalized Gaifman Conjecture},
  pdfauthor={Yi Zhang}
}
\usepackage[
  backend=biber,
  style=alphabetic,
  maxbibnames=15,
  maxcitenames=6,
  dateabbrev=false
]{biblatex}

\title{Counterexamples to the Generalized Gaifman Conjecture}

\author[Zhang]{Yi Zhang}
\address{Yi Zhang\\College of Intelligent Robotics and Advanced Manufacturing,
  Fudan University, 220 Handan Road, Shanghai 200433, China}
\email{pacisaury@gmail.com}

\newtheorem{theorem}{Theorem}[section]
\newtheorem{proposition}[theorem]{Proposition}

\theoremstyle{definition}

\newtheorem{question}[theorem]{Question}
\theoremstyle{remark}

\newcommand{\C}{\mathfrak C}

\newcommand{\Cutstar}{\operatorname{Cut}^{*}}

\DeclareMathOperator{\ded}{ded}

\begin{document}

\begin{abstract}
Shelah and Usvyatsov proposed the Generalized Gaifman Conjecture \cite[Conjecture~1.1]{SU25}. We give a negative answer.
\end{abstract}

\maketitle

\tableofcontents

\section{Introduction}

Fix a countable relational language \(L\), a complete first-order
\(L\)-theory \(T\), and a distinguished unary predicate \(P\in L\).
The Gaifman property says that every model of the theory of \(P\) occurs as the P-part
of some model of T.  Relative
categoricity over \(P\) means that two models with the same
\(P\)-part are isomorphic by an isomorphism fixing that part pointwise.
Gaifman's original question asks whether relative categoricity implies the
Gaifman property \cite[p.~30]{Gai74}; see also the modern formulation in
\cite[Definition~2.2]{SU25}.  The question remains open
\cite[p.~1]{Pil26}.

Shelah and Usvyatsov proposed a stronger statement.  Their Generalized
Gaifman Conjecture asserts that, for every sufficiently large regular
cardinal \(\kappa\) and every cardinal \(\lambda\geq\kappa\), failure of the
Gaifman property yields \(2^\kappa\) models of cardinality \(\lambda\),
non-isomorphic over a common \(P\)-part. See \cite[Conjecture~1.1]{SU25} for a detailed discussion. We construct a theory that gives a negative answer to this conjecture.

Let \(\C\models T\) be a monster model.  If \(M\) is an \(L\)-structure, write \(P^M\) for the
interpretation of \(P\) in \(M\), let \(M^P\) denote the induced
\(L\)-substructure with universe \(P^M\), and put $
  T^P=\operatorname{Th}(\C^P).
$


We recall the notion in \cite[p.~367]{PS85} used by Pillay and Shelah.
Let \(N\models T^P\).
\(M_0\cong_N M_1\) if \(M_0\) and \(M_1\) are isomorphic by an
isomorphism fixing \(N\) pointwise.
For cardinals \(\lambda,\mu\), define$
 I_T(\lambda,N)=
 \left|
   \left\{M\models T:|M|=\lambda,\ M^P=N\right\}/\cong_N
 \right|,$
and
$
  I_T(\lambda,\mu)=
  \sup\left\{I_T(\lambda,N):N\models T^P,\ |N|=\mu\right\}.
$
We replace their displayed maximum by a supremum.

\textbf{LLM disclosure}
The author found the proof with the assistance of ChatGPT. The author takes full responsibility for the entire content of the paper.

\section{The discrete theory}
\label{sec:selector}

\begin{theorem}
  \label{thm:discrete}
  There is a complete stable theory $T$ in the finite relational language $L=\{P,E\}$, with a distinguished unary predicate $P$, such that for every infinite cardinal $\kappa$:
  \begin{enumerate}
    \item There is some $N\models T^P$ of size $\kappa$ such that $I_T(\kappa, N) = 0$.
    \item $I_T(\kappa,\kappa) = \kappa$.
  \end{enumerate}
\end{theorem}
\begin{proof}
  Let $(A,E)$ be a countable equivalence relation having infinitely many classes of size $n$ for every $n\geq 1$, and one infinite class $C$. Choose $e\in C$, interpret $P$ as $A\setminus\{e\}$, and let $T=\operatorname{Th}(A,E,P)$.

  Let $N\models T^P$. Every extension $M\models T$ with $M^P=N$ has a unique element $e_M$ outside $N$, and $\{a\in N:E(e_M,a)\}$
  is an infinite $E$-equivalence class of $N$. Conversely, for every infinite $E$-equivalence class $D$ of $N$, extending $N$ by one element $e$ and making $e$ belong to $D$ gives a model $M_D\models T$ with $M_D^P=N$. The usual finite back-and-forth analysis of equivalence relations shows that every extension of $N$ to a model of $T$ is obtained this way. Moreover, $M_D$ and $M_{D'}$ are isomorphic over $N$ if and only if $D=D'$. Therefore
  $I_T(|N|,N)$ is the number of infinite $E$-equivalence classes of $N$.

  For every infinite $\kappa$ and every $\lambda\leq\kappa$, there is a model $N\models T^P$ of size $\kappa$ with exactly $\lambda$ infinite $E$-equivalence classes: take $\kappa$ classes of size $n$ for every $n\geq 1$, and add exactly $\lambda$ infinite classes. Taking $\lambda=0$ proves the first item. For every $N$ of size $\kappa$, the displayed equality gives $I_T(\kappa,N)\leq\kappa$, while taking $\lambda=\kappa$ attains this bound. Hence $I_T(\kappa,\kappa)=\kappa$.
\end{proof}

\begin{proposition}
  The following are immediate consequences of \autoref{thm:discrete}.
  \begin{enumerate}
    \item The Generalized Gaifman Conjecture fails.
    \item For any $\lambda\leq \kappa$, there is some $N$ of size $\kappa$, such that $I_T(\kappa, N) = \lambda$.
    \item $T$ is stable, and $P$ is stably embedded.
  \end{enumerate}
\end{proposition}
\begin{proof}
  The first item is clear. The second item follows from the calculation above. The theory is an equivalence relation with a distinguished predicate whose complement is a singleton, so it is stable. If $b\in P$ is equivalent to the unique point outside $P$, then its $E$-class is defined on $P$ by $E(x,b)$; hence $P$ is stably embedded.
\end{proof}

The theory is stable but fails
Hypothesis~2.1(ii) of Shelah--Usvyatsov \cite[Hypothesis~2.1(ii)]{SU25}. This shows that Morleyization can change that induced
theory; see also Usvyatsov's explicit
warning \cite[pp.~20--21]{Usv25}.

\section{The dense theory}
\label{sec:dense}
Using the same mechanism, we prove that $I_T(\kappa,\kappa)$ can also fall between $\kappa$ and $2^\kappa$ under some assumption of cardinal arithmetic.

We recall the definition in Chernikov--Shelah \cite[Fact~1.1]{CS16}.
Suppose that $I$ is a linear order. A cut is a pair $(A,B)$ such that $I=A\mathbin{\dot\cup}B$ and every element in $A$ is smaller than every element in $B$. Let $\operatorname{Cut}(I)$ be the collection of all such pairs. For an infinite cardinal $\kappa$, define
\[
  \ded(\kappa)=\sup\bigl\{|\operatorname{Cut}(I)|: I \text{ is a linear order and } |I|\leq\kappa\bigr\}.
\]
By \cite[Fact~1.1]{CS16}, this is equivalent to the standard definition using linear orders with a dense subset of size at most $\kappa$. If $D$ is a dense linear order, let $\Cutstar(D)$ be the set of cuts $(A,B)$ for which $A$ and $B$ are nonempty, $A$ has no greatest element, and $B$ has no least element.
\begin{theorem}
  \label{thm:dense}
  There is a complete theory in finite language, with a distinguished unary predicate $P$, such that:
  \begin{enumerate}
    \item $T$ fails the Gaifman property.
    \item for every infinite cardinal $\kappa$, $I_T(\kappa, \kappa)=\ded(\kappa)$, and for some $N\models T^P$, $\kappa<I_T(\kappa,N)\leq\ded(\kappa)$.
  \end{enumerate}
\end{theorem}
\begin{proof}
  Suppose that $L_{\mathrm{cut}}=\{P, <\}$. Now take the rational order $(\mathbb{Q},<)$, choose one point $c\in \mathbb{Q}$, interpret $P$ as $\mathbb{Q}\backslash \{c\}$. Let $M_{\mathrm{cut}}$ be the resulting structure. Now let $T_{\mathrm{cut}}= \mathrm{Th}(M_{\mathrm{cut}})$. In theory $T_{\mathrm{cut}}$, $c$ should be a nonprincipal cut.

  We first show that $T_{\mathrm{cut}}$ fails the Gaifman property over $P$. The order $(\mathbb{R},<)$ is a model of $T_{\mathrm{cut}}^P$, but it has no nonprincipal cut at which a new element can be inserted. Hence it cannot be the $P$-part of a model of $T_{\mathrm{cut}}$.

  Now we compute $I_{T_{\mathrm{cut}}}(\kappa, \kappa)$. Once the $P$-part is fixed, a model has exactly one additional element, and its position is determined by a cut. More precisely, for every $D\models T_{\mathrm{cut}}^P$,

   $ I_{T_{\mathrm{cut}}}(|D|,D)=|\Cutstar(D)|$. Taking the supremum gives $I_{T_{\mathrm{cut}}}(\kappa,\kappa)=\ded(\kappa)$. Since $\ded(\kappa)>\kappa$, there is some $N$ as we need.
\end{proof}

Mitchell proved that, for every cardinal $\kappa$ of uncountable cofinality, it is consistent that $\ded(\kappa)<2^\kappa$ \cite{Mit73}.
For a cardinal of countable cofinality, the same strict inequality is also consistent: Chernikov, Kaplan, and Shelah construct a model in which $\ded(\aleph_\omega)<(\ded(\aleph_\omega))^{\aleph_0}\leq 2^{\aleph_\omega}$ \cite[Section~6.2, especially Theorem~6.7, Claim~6.9, and Corollary~6.11]{CKS16}.

\begin{proposition}
  Assume that $\ded(\kappa)<2^\kappa$. There is a countable theory $T$ such that $T$ fails the Gaifman property, $I_T(\kappa,\kappa) <2^\kappa$ and $I_T(\kappa,\kappa) > \kappa $.
\end{proposition}
\begin{proof}
  This follows from \autoref{thm:dense} and the hypothesis that $\ded(\kappa)<2^{\kappa}$.
\end{proof}

\begin{proposition}
  For every infinite cardinal $\kappa$, there is a theory $T$ with a distinguished unary predicate $P$ that fails the Gaifman property. For this theory, there is some $N\models T^P$ of size $\kappa$ such that $I_T(\kappa, N)= \kappa^{\aleph_0}$.
\end{proposition}
\begin{proof}
  Use $T_{\mathrm{cut}}$. Let $I_\kappa=\kappa^{<\omega}$ with the lexicographic order, and let $D_\kappa=I_\kappa\times\mathbb{Q}$ with the lexicographic order. Then $D_\kappa$ is a dense linear order of size $\kappa$. Its branch cuts give at least $\kappa^{\aleph_0}$ proper nonprincipal cuts, and every cut is determined by a finite or countable sequence, so it has at most $\kappa^{\aleph_0}$ such cuts. Thus $|\Cutstar(D_\kappa)|=\kappa^{\aleph_0}$,  so $I_{T_{\mathrm{cut}}}(\kappa,D_\kappa)=\kappa^{\aleph_0}$.
\end{proof}

One can even push the upper bound to $\mathrm{ded}(\kappa)^{\aleph_0}$ closer to $2^\kappa$, and this is actually the original construction I used in \autoref{thm:discrete}.

This leads to the following question.
\begin{question}
  \label{question:fixed-bound}
  Can $I_T(\kappa,\kappa)$ be controlled by a fixed cardinal when $T$ fails the Gaifman property?
\end{question}

The preceding construction cannot address this question, since it uses compactness. A construction whose behavior is independent of the size of the model would require a structure capable of coding arbitrary information. We mention a theorem of Pillay and Shelah that might be useful.
\begin{theorem}
  \label{thm:finite-bound}
  Suppose that $T$ fails the Gaifman property and that its fixed-base extension problem is an admissible expansion problem by finitely many relation symbols in the sense of Pillay--Shelah, with ambient theory $\Gamma$ and $I_T(\kappa,\kappa)=D_1(\kappa)$. Then the following are equivalent.
  \begin{enumerate}
    \item $I_T(\kappa,\kappa)<\kappa$ for some $\kappa>|\Gamma|$.
    \item There is some $n<\omega$ such that $I_T(\lambda,\lambda)\leq n$ for every $\lambda>|\Gamma|$.
  \end{enumerate}
\end{theorem}
\begin{proof}
  This follows immediately from Pillay--Shelah \cite[Theorem~1.2 and Corollary~1.3]{PS85} under the stated identification.
\end{proof}

\section{Some comments on the Gaifman property and the dividing line}
\label{sec:dividing-lines}
Future work
\printbibliography

\end{document}